\documentclass[11pt,letterpaper]{amsart}
\usepackage[T1]{fontenc}
\usepackage[utf8]{inputenc}
\usepackage{amsmath,amssymb,amsthm}
\usepackage{mathrsfs}
\usepackage{booktabs}
\usepackage{tikz}
\usepackage{nicematrix}
\usepackage[all,cmtip]{xy}
\usepackage{microtype}
\usepackage{csquotes}
\usepackage{indentfirst}
\usepackage{xcolor}
\usepackage{url}
\usepackage{comment}
\usepackage[inline]{enumitem}
\setlist[enumerate]{
  itemsep=0pt,
  topsep=0pt,
  parsep=0pt,
  partopsep=0pt
}

\usepackage[pagebackref=false]{hyperref}
\hypersetup{
  hypertex=true,
  colorlinks=true,
  linkcolor=blue,
  anchorcolor=black,
  citecolor=red
}

\newtheorem{theorem}{Theorem}[section]
\newtheorem{corollary}[theorem]{Corollary}
\newtheorem{lemma}[theorem]{Lemma}
\newtheorem{proposition}[theorem]{Proposition}

\theoremstyle{definition}
\newtheorem{definition}[theorem]{Definition}
\newtheorem{example}[theorem]{Example}

\newtheorem{question}[theorem]{Question}

\theoremstyle{remark}
\newtheorem{remark}{Remark}

\theoremstyle{plain}   

\title[Sharp iterated-log thresholds for Q.M.E. and Q.O.E.]{Sharp iterated-logarithmic thresholds for quantitative measure and orbit equivalence between integer lattices}
\author{Chang-Hua JIAO}
\date{\today}
\thanks{\emph{Use of AI.} 
The proof strategy of Theorem \ref{Theorem_main_result_measure_equivalence_version} was partially obtained through the interaction with ChatGPT (Plus).
All AI-generated suggestions were verified and refined by the author. 
The paper was written by the author.
ChatGPT was used for language polishing and error
checking after the paper was written up.
The author takes full responsibility for the correctness of the paper.}

\begin{document}

\begin{abstract}
    We identify the sharp iterated-logarithmic thresholds at the critical exponent for quantitative measure equivalence and orbit equivalence between integer lattices.
    To be more precise, let $n>m$ be two positive integers, $\alpha$ be a positive number and $( \beta_j)_{j \geqslant 1}$ be a finitely supported sequence of non-negative numbers.
     We show that there is a quantitatively $t^{\alpha}\cdot \prod_{j \geqslant 1} ( \log^{(j)}{t} )^{-\beta_j} $-integrable measure equivalence from $\mathbb{Z}^n$ to $\mathbb{Z}^m$ if and only if either $\alpha<m/n$ or $\alpha=m/n$ and there is $j_* \geqslant 1$ such that $\beta_j= 1$ for all $1 \leqslant j < j_*$ and $\beta_{j_*} > 1$.
     Here $\log^{(j)}$ is the $j$-fold iterated logarithm.
     The same characterization holds for quantitative orbit equivalence.
     This characterization greatly strengthens the previous best-known result, due to the work of Delabie, Koivisto, Le Ma\^itre and Tessera (2022) and the work of Correia (2025), which asserts that there is a $t^{\alpha}$-integrable ($\alpha>0$) measure equivalence (or orbit equivalence) from $\mathbb{Z}^n$ to $\mathbb{Z}^m$ if and only if $\alpha<m/n$.
     In particular, our result solves, in much stronger forms, two open problems posed respectively by Delabie, Koivisto, Le Ma\^itre and Tessera, and by Naryshkin and Petrakos.
\end{abstract}

\maketitle

\tableofcontents

\section{Introduction}
\textbf{Background.}
\emph{Measure equivalence} and \emph{orbit equivalence} between countable groups (or their actions) are two fundamental notions in the fields of group theory and dynamical systems.
The classical theory of orbit equivalence goes back to the work of Dye \cite{On_groups_of_measure_preserving_transformations_1,On_groups_of_measure_preserving_transformations_2}, while measure equivalence was introduced by Gromov \cite{Asymptotic_invariants_of_infinite_groups} as a measure-theoretic analogue of quasi-isometry.
For non-amenable groups (and their actions), these relations often exhibit strong rigidity phenomena and have led to many celebrated results, see, e.g. \cite{Orbit_equivalence_rigidity,Gromov's_measure_equivalence_and_rigidity_of_higher_rank_lattices,Orbit_equivalence_rigidity_and_bounded_cohomology,Measure_equivalence_rigidity_of_the_mapping_class_group}.
In the amenable setting, however, these relations are extremely flexible (indeed trivial): the Ornstein-Weiss theorem in \cite{Ergodic_theory_of_amenable_group_actions._I._The_Rohlin_lemma} implies that any two countably infinite amenable groups are orbit equivalent, and hence measure equivalent. 
This naturally motivates the \emph{quantitative} refinements of these two notions by imposing some additional (e.g. integrability) conditions.

This perspective dates back to at least Belinskaya's Theorem \cite{Partitionings_of_a_Lebesgue_space_into_trajectories_which_may_be_defined_by_ergodic_automorphisms}, which says that two $L^1$-orbit equivalent ergodic probability measure preserving $\mathbb{Z}$-actions are indeed flip-conjugate.
In 2004, Shalom \cite{Harmonic_analysis_cohomology_and_the_large-scale_geometry_of_amenable_groups} studied $L^{\infty}$-measure equivalence for countable amenable groups and its relation to quasi-isometry.
The notion of $L^p$-measure equivalence was later introduced and studied by Bader, Furman, and Sauer \cite{Integrable_measure_equivalence_and_rigidity_of_hyperbolic_lattices}, 
and several related quantitative restrictions on orbit-equivalence cocycles and their dynamical consequences were investigated in, e.g.,\cite{Integrable_measure_equivalence_for_groups_of_polynomial_growth,Behaviour_of_entropy_under_bounded_and_integrable_orbit_equivalence,Entropy_Shannon_orbit_equivalence_and_sparse_connectivity,Entropy_virtual_Abelianness_and_Shannon_orbit_equivalence}.

A systematic framework allowing more general integrability functions was recently introduced by Delabie, Koivisto, Le Ma\^itre and Tessera \cite{Quantitative_measure_equivalence_between_amenable_groups}.
For amenable groups, they obtained rigidity results through the monotonicity of the isoperimetric profile and complementary flexibility results by constructing orbit equivalences from Følner tiling sequences. 
Their methods produced many essentially sharp examples, although in several cases a logarithmic gap remained between the obstructions and the constructions. 
The construction and the inverse problems arising from this framework were further investigated by Escalier \cite{Building_prescribed_quantitative_orbit_equivalence_with_the_ntegers}.

\medskip
\textbf{Critical exponent for integer lattices.}
To test this quantitative theory, integer lattices perhaps form the most basic and natural family.
Let $n>m$ be two positive integers.
Delabie, Koivisto, Le Ma\^itre and Tessera in \cite{Quantitative_measure_equivalence_between_amenable_groups} proved that, for each $0<\alpha<m/n$, there exists a $t^{\alpha}$-integrable orbit (and hence measure) equivalence from $\mathbb{Z}^n$ to $\mathbb{Z}^m$, while for all $\alpha>m/n$ there always do not exist $t^{\alpha}$-integrable ones.
The critical case $\alpha=m/n$ was settled by Correia \cite{On_the_absence_of_quantitatively_critical_measure_equivalence_couplings} very recently: no $t^{m/n}$-integrable measure equivalence from $\mathbb{Z}^n$ to $\mathbb{Z}^m$ exists.
This completes the classification at the power scale, but the behaviour at much finer sacles remains open.
More precisely, Theorem 6.12 in \cite{Quantitative_measure_equivalence_between_amenable_groups} says that there is an orbit equivalence (and hence measure equivalence) from $\mathbb{Z}^n$ to $\mathbb{Z}^m$ which is $t^{m/n}/ (\log{t})^{1+\epsilon} $-integrable for all $\epsilon>0$, while $t/\log t$-integrability is unknown to them.

\begin{question} [Question 6.14 in \cite{Quantitative_measure_equivalence_between_amenable_groups}] \label{Question_by_DKLMT_2022}
    Let $n>m$ be two positive integers, is $\mathbb{Z}^n$ always $(t^{m/n}/\log{t},L^{\infty})$-measure equivalent to $\mathbb{Z}^m$? What about $(t^{m/n}/\log{t},L^{0})$-measure equivalence?
\end{question}

The results mentioned above can be direcly applied to quantitative orbit equivalence between $\mathbb{Z}^d$-actions:
 a free p.m.p $\mathbb{Z}^n$-action is at best sub-$(t^{m/n},t^{n/m})$-orbit equivalent to another free p.m.p $\mathbb{Z}^m$-action.
Naryshkin and Petrakos very recently confirmed that this is indeed optimal for free $\mathbb{Z}$-odometers in \cite{Quantitative_orbit_equivalence_for_Z_odometers}.
Formally, they show that any two free $\mathbb{Z}$-odometers are sub-$(t,t)$-orbit equivalent and then asked the following question.

\begin{question}[Problem 3.6 in \cite{Quantitative_orbit_equivalence_for_Z_odometers}] \label{Question_by_Naryshkin_and_Petrakos_2026}
    Is a free $\mathbb{Z}^n$-odometer necessarily sub-$(t^{m/n},t^{n/m})$-orbit equivalent to a free $\mathbb{Z}^m$-odometer?
\end{question}

\medskip
\textbf{Main results.}
In this paper, we will give answers to both Question \ref{Question_by_DKLMT_2022} and Question \ref{Question_by_Naryshkin_and_Petrakos_2026}.
Although the answers are negative, our results will go much further.
Rather than merely constructing counterexamples, we will show that there are \textbf{no positive examples} for either question.
Indeed, the best known quantitative thresholds will be significantly improved, both for measure equivalence and for orbit equivalence between integer lattices.

We now state our main result explicitly.
The notation and definitions will be specified later in the next section.
For an integer $k \geqslant 1$ and a sufficiently large number $t>0$, we inductively define $$\log^{(1)}(t):= \log{t}, \quad \log^{(k+1)}(t):= \log { \left(\log^{(k)}{t}\right)}, \quad \text{and} \quad \Lambda_k(t)= \prod_{j=1}^k \log^{(j)}{t}.$$
By convention, we let $\Lambda_0(t):=1$.

Since we are concerned only with the values close to $+\infty$, for an integer $k \geqslant 1$ and a real number $\alpha>0$, we put 
$$\varpi^{(k)}_{\alpha}(t)=\begin{cases}
t^{\alpha}/\Lambda_k(t), & \quad t \geqslant T \\
T^{\alpha}/\Lambda_k(T), & 0 \leqslant
t < T=T
\end{cases}$$
where $T:=T_{\alpha,k}>0$ is chosen to make $\varpi_{\alpha,k} : \mathbb{R}_{\geqslant 0} \to \mathbb{R}$ continuous, positive and increasing (see Subsection 2.3 for details).
Note that $\varpi^{(0)}_{\alpha}(t)= t^{\alpha}$ for all sufficiently large  $t \geqslant 0$.
Moreover, for any given $\alpha>0$ and $k \geqslant 0$, we have $$ \varpi^{(k+1)}_{\alpha} = o\left(\varpi^{(k)}_{\alpha}\right), \quad t \to +\infty. $$

Our first main result is as follows.

\begin{theorem} \label{Theorem_main_result_measure_equivalence_version}
  Let $n,m,k$ be positive integers such that $n>m$.
  Then any measure equivalence from $\mathbb{Z}^n$ to $\mathbb{Z}^m$ is not $\varpi^{(k)}_{m/n}$-integrable.
\end{theorem}

Theorem \ref{Theorem_main_result_measure_equivalence_version} greatly strengthens previous results.
By taking $k=0$, we obtain the result of Correia in \cite{On_the_absence_of_quantitatively_critical_measure_equivalence_couplings} for integer lattices, which asserts that any measure equivalence from $\mathbb{Z}^n$ to $\mathbb{Z}^m$ is not $t^{m/n}$-integrable when $n>m$.
Taking $k=1$ gives a complete answer to Question \ref{Question_by_DKLMT_2022}.
Namely, $\mathbb{Z}^n$ is neither $(t^{m/n}/\log{t},L^{\infty})$-measure equivalent nor $(t^{m/n}/\log{t},L^0)$-measure equivalent to $\mathbb{Z}^m$ when $n>m$.

Since quantitative orbit equivalence is stronger than quantitative measure equivalence (see Proposition \ref{Proposition_QOE_is_stronger_than_QME}), we immediately obtain the following corollary.

\begin{corollary} \label{Corollary_main_result_orbit_equivalence_version}
    Let $n>m$ be two positive integers and $\mathbb{Z}^n \curvearrowright (X, \mu)$ and $\mathbb{Z}^m \curvearrowright (Y, \nu)$ be two free p.m.p. actions.
Suppose that $\Phi: X \to Y$ is an orbit equivalence.
Then $\Phi$ is not $\varpi^{(k)}_{m/n}$-integrable.
Namely, $\mathbb{Z}^n$ cannot be $\left(\varpi^{(k)}_{m/n},L^0\right)$-orbit equivalent to $\mathbb{Z}^m$ when $n>m$.
\end{corollary}

 Corollary \ref{Corollary_main_result_orbit_equivalence_version} yields a negative answer to Question \ref{Question_by_Naryshkin_and_Petrakos_2026}:
 a free p.m.p. $\mathbb{Z}^n$-action cannot be sub-$(t^{m/n}, t^{n/m})$-orbit equivalent to a free p.m.p. $\mathbb{Z}^m$-action for $n \neq m$.
 Note that this answer is of full generality, in the sense that it is valid for \textbf{all free p.m.p. $\mathbb{Z}^d$-actions} rather than just for free $\mathbb{Z}^d$-odometers, the special kind asked for in Question \ref{Question_by_Naryshkin_and_Petrakos_2026}.

Theorem \ref{Theorem_main_result_measure_equivalence_version} is logarithmically sharp in the following sense.
For an integer $k \geqslant 1$, and two real numbers $\alpha>0$ and $\epsilon>0$, we put
$$\varpi^{(k)}_{\alpha, \epsilon}(t)=\begin{cases}
\left(t^\alpha/\Lambda_{k}(t)\right) \cdot \left(\log^{(k)}{t}\right)^{-\epsilon}, & \quad t \geqslant T \\
\left(T^\alpha/\Lambda_{k}(T)\right) \cdot \left(\log^{(k)}{T}\right)^{-\epsilon}, & 0 \leqslant t < T
\end{cases}$$
where $T:=T_{\alpha,k, \epsilon}>0$ is chosen to make $\varpi^{(k)}_{\alpha, \epsilon} : \mathbb{R}_{\geqslant 0} \to \mathbb{R}$ continuous, positive and increasing (see Subsection 2.3 for details).

\begin{proposition} \label{Proposition_quantitative_orbit_equivalence_for_integer_lattices}
    Let $n,m$ be two positive integers.
    Then $\mathbb{Z}^n$ is always $\left(\varpi^{(k)}_{m/n, \epsilon}, \varpi^{(k)}_{n/m, \epsilon}\right)$-orbit equivalent (and hence $\left(\varpi^{(k)}_{m/n, \epsilon}, \varpi^{(k)}_{n/m, \epsilon}\right)$-measure equivalent) to $\mathbb{Z}^m$ for every real number $\epsilon>0$ and every integer $k \geqslant 1$.
\end{proposition}

In view of Proposition \ref{Proposition_QOE_is_stronger_than_QME}, the above Proposition \ref{Proposition_quantitative_orbit_equivalence_for_integer_lattices} follows immediately from the following stronger result, which is our second main result.
Recall that for a positive integer $d$ and a prime $p$, the $d$-dimensional $p$-adic adding machine is the $\mathbb{Z}^d$-action on $\mathbb{Z}_p^d$ by addition (see Subsection 5.1 for more details).

\begin{theorem} \label{Theorem_quantitative_orbit_equivalence_for_p_adic_adding_machines}
    Let $n,m$ be two positive integers and $p$ be a prime.
    Then there is an orbit equivalence from the $n$-dimensional $p$-adic adding machine to the $m$-dimensional one, which is $\left(\varpi^{(k)}_{m/n, \epsilon}, \varpi^{(k)}_{n/m, \epsilon}\right)$-integrable for every real number $\epsilon>0$ and every integer $k \geqslant 1$.
\end{theorem}
Our result clearly strengthens Theorem 6.12 in \cite{Quantitative_measure_equivalence_between_amenable_groups}.
Moreover, these results should be compared with results in \cite{Building_prescribed_quantitative_orbit_equivalence_with_the_ntegers}.
In fact, our results improve its Theorem 1.7 and its Corollary 1.8 in the case $\rho= \exp$, since our orbit equivalence is valid uniformly for all $\epsilon >0$.

Combining our two main results above, we obtain a complete picture for iterated-logarithmic integrability for a measure (or orbit) equivalence from a higher dimensional integer lattice to a lower dimensional integer lattice.
To be more precise, for $\alpha>0$ and a finitely supported (i.e., only finitely many terms are non-zero) sequence $\hat{\beta}=(\beta_j)_{j \geqslant 1}$ with each $\beta_j \geqslant 0$, we put 
$$\varPi_{\alpha, \hat{\beta}}(t)= t^\alpha \cdot \prod_{j=1}^{\infty} \left( \log^{(j)}t \right)^{-\beta_j}$$
for all sufficiently large $t>0$.

\begin{corollary}
    Let $n>m$ be two positive integers.
    Suppose that $\alpha>0$ and $\hat{\beta}=(\beta_j)_{j \geqslant 1}$ is a finitely supported sequence with each $\beta_j \geqslant 0$.
    Then there is a $\varPi_{\alpha, \hat{\beta}}$-integrable measure equivalence from $\mathbb{Z}^n$ to $\mathbb{Z}^m$ if and only if one of the following two conditions holds:
    \begin{enumerate}
        \item $\alpha<m/n$;
        \item $\alpha=m/n$ and there is $j_* \geqslant 1$ such that $\beta_j= 1$ for all $1 \leqslant j < j_*$ and $\beta_{j_*} > 1$.
    \end{enumerate}
    The same assertion holds when one replaces measure equivalence by orbit equivalence.
\end{corollary}

\medskip
\noindent\textbf{Overview of the proofs.}
The proof of Theorem \ref{Theorem_main_result_measure_equivalence_version} mainly relies on two lemmas (Lemma \ref{Lemma_Estimating_sum_over_edges} and Lemma \ref{Lemma_many_x_has_large_image}). 
To prove the first one, we count bridges in the pullback of a random partition of the target lattice into cubes. 
The edge-isoperimetric inequality gives a lower bound for the number of bridges. 
For each edge, the probability that it is a bridge is at most the distance between the images of its endpoints divided by the side length of the cubes. 
Summing the resulting estimates over dyadic scales with suitable weights then gives Lemma \ref{Lemma_Estimating_sum_over_edges}, which gives a lower bound for the $\varpi^{(k)}_{m/n}$-sum over edges in terms of the image cardinality and its $(k+1)$-fold iterated logarithm.

Lemma \ref{Lemma_many_x_has_large_image} can be achieved by bounding the sizes of cocycle fibres on the hitting time set to a suitable part of the fundamental domain. It says that, for each finite subset of the source group, its image under the cocycle map has comparable cardinality on a set of uniformly positive measure.
Applying Lemma \ref{Lemma_Estimating_sum_over_edges} to these maps and integrating, we obtain the failure of $\varpi^{(k)}_{m/n}$-integrability by the cocycle identity.

To prove Theorem \ref{Theorem_quantitative_orbit_equivalence_for_p_adic_adding_machines}, we construct an orbit equivalence between $p$-adic adding machines by interleaving and regrouping their digits.
Considering the distribution of the number of initial digits equal to $p-1$, the required integrability therefore reduces to convergence of the corresponding iterated-logarithmic series.
Note that the same digit rearrangement works simultaneously for every $k \geqslant 1$ and $\epsilon>0$.

\medskip
\textbf{Further questions.}
Firstly, although in general Question \ref{Question_by_Naryshkin_and_Petrakos_2026} has a negative answer, but one can again narrow the scope to consider sub-$(t,t)$-orbit equivalence for two free $\mathbb{Z}^d$-odometers (of the same dimension $d:=m=n$).
Our method in the present paper relies heavily on the difference in dimension and hence cannot be applied directly to this case.
So the narrowed case of Question \ref{Question_by_Naryshkin_and_Petrakos_2026} remains open.

Secondly, note that the $d$-dimensional $p$-adic adding machine is automatically a free $\mathbb{Z}^d$-odometer.
So it leads to the following refinement of Question \ref{Question_by_Naryshkin_and_Petrakos_2026}, in view of Theorem \ref{Theorem_quantitative_orbit_equivalence_for_p_adic_adding_machines}.
\begin{question}
    Let $n, m$ be two positive integers and $\epsilon>0$.
    Is a free $\mathbb{Z}^n$-odometer necessarily $\left(\varpi^{(k)}_{m/n, \epsilon},\varpi^{(k)}_{n/m, \epsilon}\right)$-orbit equivalent to a free $\mathbb{Z}^m$-odometer for all $k \geqslant 1$?
\end{question}

Although Theorem \ref{Theorem_quantitative_orbit_equivalence_for_p_adic_adding_machines} focuses only on $p$-adic adding machines, we believe that our method can be directly adapted to address Question 1.8 in more further cases of odometers.
However, the general case remains open to the author.
Note that our results indicate that they cannot be sub-$\left(\varpi^{(k)}_{m/n},\varpi^{(k)}_{n/m}\right)$-orbit equivalent.

\medskip
\textbf{Organization.} The rest of the paper is structured as follows.
In Section 2, we collect preliminaries on (quantitative) measure equivalence and orbit equivalences and include some elementary facts about the function $\varpi^{(k)}_{\alpha}$.
In Section 3, we establish some combinatorial lemmas, which helps to prove Theorem \ref{Theorem_main_result_measure_equivalence_version}.
Section 4 is devoted to the proof of Theorem \ref{Theorem_main_result_measure_equivalence_version}.
In the last section, we study the quantitative orbit equivalence between $p$-adic adding machines and prove Theorem \ref{Theorem_quantitative_orbit_equivalence_for_p_adic_adding_machines}.

\section{Preliminaries}

\subsection{Measure equivalence and orbit equivalence}

A \emph{standard Borel space} $(X, \mathcal{B})$ is a measurable space which is Borel isomorphic to a Polish space (i.e., a complete and separable metric space) equipped with its Borel $\sigma$-algebra.
A \emph{standard measure space} $(X, \mathcal{B}, \mu)$ consists of a standard Borel space $(X, \mathcal{B})$ and a non-zero $\sigma$-finite measure $\mu$ on it.
We sometimes simply write $(X, \mu)$ for a standard measure space.

Let $G$ be a countable discrete group and $(X, \mathcal{B}, \mu)$ be a standard measure space.
A (measurable) group action $G \curvearrowright (X, \mathcal{B}, \mu)$ is said to be \emph{measure-preserving} if $$\mu(g^{-1}B)= \mu(B), \quad \forall B \in \mathcal{B}, g \in G .$$ 
A \emph{probability measure-preserving (p.m.p.)} action is a measure-preserving action on a standard probability measure space.
An action $G \curvearrowright (X, \mu)$ is said to be \emph{(essentially) free} if for $\mu$-a.e. $x \in X$, we have $$g x \neq x, \quad \forall g \in G \setminus \{e_G\}.$$

\begin{definition}
    Let $G \curvearrowright (X, \mu)$ be a group action on a standard measure space.
    A \emph{fundamental domain} for $G \curvearrowright (X, \mu)$ is a Borel subset $X_G \subset X$ such that for $\mu$-a.e. $x \in X$ we have  $$|X_G \cap  Gx| = 1, $$
    where $Gx:=\{gx: g \in G\}$ is the $G$-orbit of $x$.
    The action $G \curvearrowright (X, \mu)$ is said to be \emph{smooth} if it admits a fundamental domain.
\end{definition}

\begin{definition}
    Let $G,H$ be two countable groups.
    A \emph{measure equivalence coupling} from $G$ to $H$ is a quadruple $\mathcal{C}=(X, X_G, X_H, \mu)$ where $(X,\mu)$ is a standard measure space equipped with two free, measure-preserving, and smooth actions by $G$ and $H$ such that
\begin{enumerate}
    \item the two actions are commuting in the sense that for $\mu$-a.e. $x \in X$ we have  $$hgx=ghx, \quad \forall g \in G,h \in H; $$
    \item $X_G$ is a fixed fundamental domain for $G \curvearrowright (X, \mu)$ and $X_H$ is a fixed fundamental domain for $H \curvearrowright (X, \mu)$;
    \item $\mu(X_G) < \infty$ and $\mu(X_H) < \infty$.
\end{enumerate}
If such a coupling exists, $G$ and $H$ are said to be \emph{measure equivalent}.
\end{definition}

It is obvious that $\mathcal{C}=(X, X_G, X_H, \mu)$ is a measure equivalence coupling from $G$ to $H$ if and only if $\mathcal{C}^{-1}:=(X, X_H, X_G, \mu)$ is a measure equivalence coupling from $H$ to $G$, with the same actions by $G$ and $H$.
Now given a measure equivalence coupling $\mathcal{C}=(X, X_G, X_H, \mu)$, the \emph{induced action} (with respect to $\mathcal{C}$) is defined to be $G \curvearrowright (X_H, \mu|_{X_H})$ (where the action is denoted by $g \star x$ ) such that for $\mu|_{X_H}$-a.e. $x \in X$ we have  $$\{ g \star x \}= X_H \cap H(gx), \quad \forall g \in G.$$
Note that this induced action is measure preserving.

\begin{definition}
    Given a measure equivalence coupling $\mathcal{C}=(X, X_G, X_H, \mu)$ from $G$ to $H$, the \emph{cocycle} with respect to $\mathcal{C}$ is defined to be the map $\kappa^{\mathcal{C}}: G \times X_H \to H$ such that for $\mu|_{X_H}$-a.e. $x \in X_H$, we have $$\kappa^{\mathcal{C}}(g,x) (gx)= g \star x, \quad \forall g \in G.$$
\end{definition}

The definition makes sense since the freeness of the action $H \curvearrowright (X, \mu)$ is assumed.
Note that the cocycle $\kappa^{\mathcal{C}}$ is clearly measurable (with Borel measurable structures on the two countable discrete groups) and we have the following \emph{cocycle identity}:
\begin{equation} \label{Equation_cocycle_identity_for_measure_equivalence_cocyle}
    \kappa^{\mathcal{C}}(g_2g_1,x)= \kappa^{\mathcal{C}}(g_2,g_1\star x) \kappa^{\mathcal{C}}(g_1,x), \quad  \text{ for } \mathrm{a.e.}\  x \in X_H \text{ and all } g_1,g_2 \in G.
\end{equation}

We now briefly introduce orbit equivalence.

\begin{definition}
   Let $G \curvearrowright (X, \mu)$ and $H \curvearrowright (Y, \nu)$ be two p.m.p. actions.
    An \emph{orbit equivalence} from $G \curvearrowright (X, \mu)$ to $H \curvearrowright (Y, \nu)$ is a measure isomorphism $\Phi: X \to Y$ that preserves the orbit structure; that is, for $\mu$-a.e. $x \in X$, we have  $$\Phi(Gx)=H\Phi(x).$$
    These two actions are said to be \emph{orbit equivalent} if such an orbit equivalence exists.
    Two groups are said to be \emph{orbit equivalent} if they admit orbit equivalent actions.
\end{definition}

It is easy to see that $\Phi^{-1}: Y \to X$ is an orbit equivalence from $H \curvearrowright (Y, \nu)$ to $G \curvearrowright (X, \mu)$ if and only if $\Phi$ is an orbit equivalence from $G \curvearrowright (X, \mu)$ to $H \curvearrowright (Y, \nu)$.

\begin{definition}
 If two free p.m.p. actions $G \curvearrowright (X, \mu)$ and $H \curvearrowright (Y, \nu)$ are orbit equivalent via $\Phi: X \to Y$, we can define $\kappa^{\Phi}:G \times X \to H$ such that for a.e. $x \in X$, we have 
$$\Phi(gx)=\kappa^{\Phi}(g,x)\Phi(x), \quad \forall g \in G.$$ 
The map $\kappa^{\Phi}$ is called the \emph{(orbit) cocycle} with respect to $\Phi$.
\end{definition}

The cocycle $\kappa^{\Phi}$ is well-defined (in measure-theoretic sense) since the $H$-action is free.
Clearly, $\kappa^{\Phi}$ is  measurable (with Borel measurable structures on the two countable discrete groups) and we also have the cocycle identity:
\begin{equation} \label{Equation_cocycle_identity_for_orbit_equivalence_cocycle}
    \kappa^{\Phi}(g_2g_1,x)= \kappa^{\Phi}(g_2,g_1x) \kappa^{\Phi}(g_1,x), \quad  \text{ for } \mathrm{a.e.}\  x \in X \text{ and all } g_1,g_2 \in G.
\end{equation}

We note that $G$ and $H$ are measure equivalent if there are two free p.m.p. actions $G \curvearrowright (X, \mu)$ and $H \curvearrowright (Y, \nu)$ which are orbit equivalent.

\subsection{Quantitative versions of the equivalences}
Suppose $G$ is a finitely generated group; say $G=\langle S \rangle$ for a finite subset $S \subset G$.
The \emph{word length} of $g \in G$ (w.r.t. the finite generating set $S$) is defined to be $$|g|_S:= \min\{n \geqslant 0: g=s_1 s_2 \cdots s_n \text{ with each } s_i \in S \cup S^{-1}\}.$$
By convention, $|e_G|_S=0$.
This naturally induces a \emph{word-length metric} (w.r.t. $S$) on $G$ defined by $$\mathrm{d}_S(g_1,g_2):= |g_1g_2^{-1}|_S, \quad \forall g_1,g_2 \in G.$$
For two different finite generating sets $S, S' \subset G$, the corresponding word-length metrics are bi-Lipschitz equivalent; that is, there is $C>0$ such that $$C^{-1} \cdot \mathrm{d}_S(g_1,g_2) \leqslant \mathrm{d}_{S'}(g_1,g_2) \leqslant C \cdot \mathrm{d}_{S}(g_1,g_2), \quad \forall g_1,g_2 \in G.$$
So in the sequel we will denote by $| \cdot |_G$ and $\mathrm{d}_G$ the word length and the word-length metric on $G$ respectively, since our considerations remain unchanged under bi-Lipschitz transformation.
We will mainly use the following word length for integer lattices.

\begin{example}
We will use $|\cdot|$ rather than $|\cdot|_{\mathbb{Z}^d}$ to denote the word length on $\mathbb{Z}^d$, which is, more precisely, the sum of the absolute values of all coordinates.
Moreover, let $\mathbf{e}_j$ be the $j$-th standard basis vector of $\mathbb{Z}^d$ and fix the generating set $S=\{ \mathbf{e}_1, \cdots, \mathbf{e}_d\}$ of $\mathbb{Z}^d$ .
    Then we have $|\mathbf{v}|=|\mathbf{v}|_S=\sum_{j=1}^d |v_j|$ where $\mathbf{v}=(v_1,v_2, \cdots, v_d)$.
\end{example}

From now on, we assume that our groups are always finitely generated unless otherwise stated.

\begin{definition} [See, for example, Definition 2.4 in \cite{On_the_absence_of_quantitatively_critical_measure_equivalence_couplings}]
   Let $\omega: \mathbb{R}_{\geqslant 0} \to \mathbb{R}_{\geqslant 0}$ be a non-decreasing function and let $G,H$ be two groups.
   A measure equivalence coupling $\mathcal{C}=(X, X_G, X_H, \mu)$ from $G$ to $H$ is said to be  $\omega$-\emph{integrable} if for every $g \in G$, there is $c_g>0$ such that
   \begin{equation} \label{equation_omega-integrability_for_measure_equivalence_coupling}
   \int_{X_H} \omega \left( \frac{|\kappa^{\mathcal{C}}(g,x)|_H}{c_g}\right) \mathrm{d} \mu(x)< \infty.
   \end{equation}
   Moreover, $\mathcal{C}$ is said to be $(\omega_1, \omega_2)$-\emph{integrable} for two non-decreasing functions $\omega_1, \omega_2: \mathbb{R}_{\geqslant 0} \to \mathbb{R}_{\geqslant 0}$, if $\mathcal{C}$ is $\omega_1$-integrable and $\mathcal{C}^{-1}$ is $\omega_2$-integrable.
   We say that a group is (\emph{quantitatively}) $(\omega_1, \omega_2)$-\emph{measure equivalent} to another group if there is a $(\omega_1, \omega_2)$-integrable measure equivalence coupling from the first group to the second group.
\end{definition}

\begin{definition} [See, for example, Definition 2.1 in \cite{Quantitative_orbit_equivalence_for_Z_odometers} or Definition 1.3 in \cite{Building_prescribed_quantitative_orbit_equivalence_with_the_ntegers}]
    Let $\omega: \mathbb{R}_{\geqslant 0} \to \mathbb{R}_{\geqslant 0}$ be a non-decreasing function and let $G \curvearrowright (X, \mu)$ and $H \curvearrowright (Y, \nu)$ be two free p.m.p. actions.
   An orbit equivalence $\Phi: X \to Y$ is said to be  $\omega$-\emph{integrable} if for every $g \in G$, there is $c_g>0$ such that
   \begin{equation} \label{equation_omega-integrability_for_orbit_equivalence}
   \int_{X} \omega \left( \frac{|\kappa^{\Phi}(g,x)|_H}{c_g}\right) \mathrm{d} \mu(x)< \infty.
   \end{equation}
   Moreover, $\Phi$ is said to be $(\omega_1, \omega_2)$-\emph{integrable} for two non-decreasing functions $\omega_1, \omega_2: \mathbb{R}_{\geqslant 0} \to \mathbb{R}_{\geqslant 0}$, if $\Phi$ is $\omega_1$-integrable and $\Phi^{-1}$ is $\omega_2$-integrable.
   We say that a free p.m.p. action is (\emph{quantitatively}) $(\omega_1, \omega_2)$-\emph{orbit equivalent} to the other one if there is a $(\omega_1, \omega_2)$-integrable orbit equivalence from the first action to the second action.
   A group $G$ is said to be (\emph{quantitatively}) $(\omega_1, \omega_2)$-\emph{orbit equivalent} to a group $H$ if some free p.m.p. $G$-action is $(\omega_1, \omega_2)$-orbit equivalent to a free p.m.p. $H$-action.
   \end{definition}

    \begin{remark} \label{Remark_checking_integrability_for_generating_set}
    According to Proposition 2.22 in \cite{Quantitative_measure_equivalence_between_amenable_groups}, to see $\omega$-integrability for a measure equivalence or orbit equivalence, one only needs to check (\ref{equation_omega-integrability_for_measure_equivalence_coupling}) or (\ref{equation_omega-integrability_for_orbit_equivalence}) for a generating set of $G$.
   \end{remark}

   For two functions $f_1,f_2 : \mathbb{R}_{\geqslant 0} \to \mathbb{R}_{\geqslant 0}$, when we write $f_1=o(f_2)$, we mean $f_1(t)=o(f_2(t))$ as $t \to +\infty$; that is, $\lim_{t \to +\infty}f_1(t)/f_2(t) = 0$.
   The notation will be simplified as follows:
   (where $\kappa$ is a cocycle either for some measure equivalence coupling or for some orbit equivalence)
   \begin{enumerate}
        \item $\kappa$ is $L^{0}$-integrable if there is no restriction (since it is always measurable);
        \item $\kappa$ is $L^{\infty}$-integrable if $\kappa$ is essentially bounded;
        \item By sub-$(\omega_1, \omega_2)$-measure equivalence (resp. sub-$(\omega_1, \omega_2)$-orbit equivalence), we mean $(\omega'_1, \omega'_2)$-measure equivalence (resp. $(\omega'_1, \omega'_2)$-orbit equivalence) for all $\omega'_i=o(\omega_i)$, $i=1,2$.
    \end{enumerate}
    With the above notation, $(L^0,L^0)$-measure equivalence (resp. $(L^0,L^0)$-orbit equivalence) is nothing but measure equivalence (resp. orbit equivalence).

    The relation of (quantitative) measure equivalence and (quantitative) orbit equivalence is revealed by the following result, which is the quantitative counterpart of $\mathbf{P}_{\mathrm{ME}} \mathbf{5}$ in \cite{Examples_of_groups_that_are_measure_equivalent_to_the_free_group} (see also Theorem 3.3 in \cite{Orbit_equivalence_rigidity}).
    Note that quantitative orbit equivalence is stronger than quantitative measure equivalence (with the same integrability conditions).

    \begin{proposition} [See, for example, Remark 2.7 in \cite{On_quantitative_orbit_equivalence_for_lamplighter-like_groups}] \label{Proposition_QOE_is_stronger_than_QME}
    Two groups $G,H$ are $(\omega_1, \omega_2)$-orbit equivalence if and only if there is a $(\omega_1, \omega_2)$-integrable measure equivalence coupling $(X, X_G, X_H, \mu)$ with $X_G= X_H$.
    Here $\omega_i$ can be replaced by $L^0$, $L^{\infty}$ introduced above.
    \end{proposition}

\subsection{Elementary properties of $\varpi^{(k)}_{\alpha}$}

Here we establish some facts about $\varpi^{(k)}_{\alpha}$.
Although the proofs are all elementary, we include them here for completeness and the reader's convenience.
By abuse of notation, we write $\varpi^{(k)}_{\alpha,0}=\varpi^{(k)}_{\alpha}$ .

\begin{proposition}
    Given two numbers $\alpha>0$, $\epsilon \geqslant 0$ and an integer $k \geqslant 1$, there is a $T:=T_{\alpha, k, \epsilon}>0$ such that $ f_{\alpha,k, \epsilon}(t):= t^\alpha \cdot \left(\Lambda_{k-1}(t) \cdot (\log^{(k)}{t})^{1+\epsilon} \right)^{-1}$ is continuous, positive, and strictly increasing on the interval $[T, +\infty)$.
\end{proposition}

\begin{proof}
    We will work with fixed $\alpha, k, \epsilon$ and omit all subscripts w.r.t. these quantities.
    We first take a sufficiently large $T'>0$ such that $\log^{(k)}(T')$ is defined and positive.
    Then $f$ is positive on $[T', +\infty)$.
    To see monotonicity, we compute the (logarithmic) derivative:
     $$ \frac{\mathrm{d}f(t)}{\mathrm{d} t} \cdot \frac{1}{f(t)} =\frac{\mathrm{d} \left(\log{f} (t)\right)}{\mathrm{d}t}=\frac{1}{t}  \left( \alpha-  \sum_{j=1}^{k} \frac{1}{\Lambda_j(t)} -  \frac{\epsilon}{\Lambda_k(t)} \right), \quad \forall t >T'.$$
     So we are done by taking $T>T'$ such that $$\sum_{j=1}^k \frac{1}{\Lambda_j(t)} + \frac{\epsilon}{\Lambda_k(t)} \leqslant \sum_{j=1}^{k} \frac{1}{\Lambda_j(T)} + \frac{\epsilon}{\Lambda_k(T)}<\alpha, \quad \forall t\geqslant T>0.$$
     The continuity of $f$ on $[T, +\infty)$ is trivial.
\end{proof}

\begin{proposition}\label{Proposition_the_same_integrability_up_to_constant}
        Let $(X,\mu)$ be a finite measure space and $f: X \to \mathbb{R}_{\geqslant 0}$ a measurable function.
        Then for any given numbers $\alpha, c_1, c_2>0$ and integer $k \geqslant 1$, we have $$\int_X \varpi^{(k)}_{\alpha}\left(f(x)/c_1\right) \mathrm{d}\mu(x) < \infty \iff \int_X \varpi^{(k)}_{\alpha} \left(f(x)/c_2 \right) \mathrm{d}\mu(x) < \infty.$$
\end{proposition}

\begin{proof}
        This is a direct consequence of the fact that for any fixed number $\alpha, c>0$ and integer $k \geqslant 1$, there is $d>0$ such that 
        \begin{equation} \label{Equation_aim}
        \varpi^{(k)}_{\alpha}(ct) \leqslant d \cdot \varpi^{(k)}_{\alpha}(t), \quad \forall t \geqslant 0. 
        \end{equation}
        To see (\ref{Equation_aim}), we note that
        $$ \lim_{t \to +\infty} \frac{\varpi^{(k)}_{\alpha}(ct)}{\varpi^{(k)}_{\alpha}(t)} = c^{\alpha} \cdot \lim_{t \to +\infty} \prod_{j=1}^{k} \frac{\log^{(j)}(ct)}{\log^{(j)}(t)}= c^{\alpha} \in (0,+\infty).$$
       We are done since $\varpi^{(k)}_{\alpha}$ is bounded on any compact subset.
\end{proof}

The following result plays an important role in proving Lemma \ref{Lemma_Estimating_sum_over_edges}.

\begin{lemma} \label{Lemma_elementary_lemma_lower_bound_of_varpi_p}
    Given a number $0<\alpha<1$ and an integer $k \geqslant 1$, there is a number $A_{\alpha,k}>0$ and an integer $j_{\alpha,k}>0$ depending only on $\alpha$ and $k$ such that $\log^{(k)}(2^j)>0$ for all $j \geqslant j_{\alpha, k}$ and 
     $$\varpi^{(k)}_{\alpha}(t) \geqslant A_{\alpha,k} \cdot \sum_{j=j_{\alpha,k}}^{\infty}  \frac{2^{j\alpha}}{\Lambda_{k}(2^j)} \cdot \min\left\{\frac{t}{2^j}, 1\right\} , \quad \forall t \geqslant 0.$$
\end{lemma}

\begin{proof}
    We now fix $k$ and $\alpha$.
    Put $a_j:=\frac{2^{j\alpha}}{\Lambda_{k}(2^j)}>0$ for each (sufficiently large) $j$ and we see 
    $$\lim_{j \to \infty} \frac{a_{j}}{a_{j+1}}= \lim_{j \to \infty} 2^{-\alpha} \frac{\Lambda_k(2^{j+1})}{\Lambda_k(2^{j})} = 2^{-\alpha} \in (0,1).$$
    So there is an integer $j_{\alpha, k}>0$ such that $2^{j_{\alpha,k}} \geqslant T_{\alpha, k}$ and $$a_j \leqslant \frac{1+2^{-\alpha}}{2} \cdot a_{j+1}, \quad \forall j \geqslant j_{\alpha,k}.$$
     This then implies that there is $C_{\alpha,k}>0$ (independent of $J$) such that 
     \begin{equation} \label{Equation_need}
        \sum_{j=j_{\alpha,k}}^{J}  \frac{2^{j\alpha}}{\Lambda_{k}(2^j)} = \sum_{j=j_{\alpha,k}}^{J} a_j \leqslant C_{\alpha,k} \cdot a_J= C_{\alpha,k} \cdot \frac{2^{J\alpha}}{\Lambda_{k}(2^J)}, \quad \forall J \geqslant j_{\alpha,k}.  
     \end{equation}

     Now if $0 \leqslant t < 2^{j_{\alpha,k}}$, we see 
     $$\sum_{j=j_{\alpha,k}}^{\infty}  \frac{2^{j\alpha}}{\Lambda_{k}(2^j)} \cdot \min\left\{\frac{t}{2^j}, 1\right\}= t \cdot \sum_{j=j_{\alpha,k}}^{\infty}  \frac{2^{j(\alpha-1)}}{\Lambda_{k}(2^j)} \leqslant 2^{j_{\alpha,k}}\cdot \sum_{j=j_{\alpha,k}}^{\infty}  \frac{2^{j(\alpha-1)}}{\Lambda_{k}(2^j)} .$$
     Note that $2^{j_{\alpha,k}}\cdot \sum_{j=j_{\alpha,k}}^{\infty}  \frac{2^{j(\alpha-1)}}{\Lambda_{k}(2^j)}$ converges since $\alpha<1$.
     If we denote by $B_{\alpha,k}>0$ the sum of this series, then for all $t \in [0, 2^{j_{\alpha,k}})$, we have
     \begin{equation} \label{Equation_inequality_1}
        \varpi^{(k)}_{\alpha}(t) \geqslant \varpi^{(k)}_{\alpha}(0) \geqslant \varpi^{(k)}_{\alpha}(0) \cdot B_{\alpha, k}^{-1} \cdot \left( \sum_{j=j_{\alpha,k}}^{\infty}  \frac{2^{j\alpha}}{\Lambda_{k}(2^j)} \cdot  \min\left\{\frac{t}{2^j}, 1\right\} \right). 
     \end{equation}

    Now we suppose $2^J \leqslant t< 2^{J+1}$ for some $J \geqslant j_{\alpha,k}$.
    Then we have \begin{align*}
       \sum_{j=j_{\alpha,k}}^{\infty}  \frac{2^{j\alpha}}{\Lambda_{k}(2^j)} \cdot \min\left\{\frac{t}{2^j}, 1\right\} 
       & = \sum_{j=j_{\alpha,k}}^{J} \frac{2^{j\alpha}} {\Lambda_{k}(2^j)} + t \cdot \sum_{j=J+1}^{\infty} \frac{2^{j(\alpha-1)}} {\Lambda_{k}(2^j)}\\
     \text{by (\ref{Equation_need}) }  & \leqslant C_{\alpha,k} \cdot \frac{2^{J\alpha}}{\Lambda_{k}(2^J)} + \frac{t}{\Lambda_{k}(2^{J+1})} \cdot \sum_{j=J+1}^{\infty} 2^{j(\alpha-1)} \\
   \text{(since  $t\leqslant 2^{J+1}$) }  & \leqslant C_{\alpha,k} \cdot \frac{2^{J\alpha}}{\Lambda_{k}(2^J)} +  \frac{2^{J+1}}{\Lambda_{k}(2^{J+1})} \cdot \sum_{j=J+1}^{\infty} 2^{j(\alpha-1)}.\\
     & = C_{\alpha,k} \cdot \frac{2^{J\alpha}}{\Lambda_{k}(2^J)} + \frac{2^{(J+1) \alpha}}{\Lambda_{k}(2^{J+1})} \cdot \sum_{j=J+1}^{\infty} 2^{(j-J-1)(\alpha-1)}\\
     & = C_{\alpha,k} \cdot  a_J + D_{\alpha,k} \cdot a_{J+1}
    \end{align*}
where $D_{\alpha,k}=\sum_{j=0}^{\infty} 2^{j(\alpha-1)}>0$ (it converges since $\alpha<1$).
Since $\{a_{J+1} / a_{J}\}_{J \geqslant j_{\alpha, k}}$ is bounded above, we see there is $E_{\alpha, k}>0$ (independent of $J$) such that for all $J \geqslant j_{\alpha, k}$, we have 
\begin{equation} \label{Equation_inequality_2}
    \sum_{j=j_{\alpha,k}}^{\infty}  \frac{2^{j\alpha}}{\Lambda_{k}(2^j)} \cdot \min\left\{\frac{t}{2^j}, 1\right\}  \leqslant E_{\alpha,k} \cdot a_J \leqslant E_{\alpha,k} \cdot \varpi^{(k)}_{\alpha} (t), \quad \forall t \in [2^J,2^{J+1}).
\end{equation}

In view of (\ref{Equation_inequality_1}) and (\ref{Equation_inequality_2}), we are done by taking $$A_{\alpha,k}:= \min\left\{  E_{\alpha,k}^{-1}, \  \varpi^{(k)}_{\alpha}(0) \cdot B_{\alpha, k}^{-1} \right\}>0.$$
\end{proof}

\section{Combinatorial lemmas in grids}

\subsection{Edge-isoperimetric inequality}
Here we recall a combinatorial result in \cite{Edge-isoperimetric_inequalities_in_the_grid} due to Bollob\'{a}s and Leader.
 For two real numbers $k<l$, we will denote by $[[k,l]]:=[k,l] \cap \mathbb{Z}$ the integer interval between $k$ and $l$ and we simply write $[[N]]:=[[0,N-1]]$ for a positive integer $N$.
For positive integers $n,N$, we consider the (undirected) graph $\mathcal{G}=\mathcal{G}_{N,n}=(V,E)=(V_{N,n} , E_{N,n})$ where $V=[[N]]^n$ and $\{\mathbf{v}, \mathbf{u}\} \in E$ if and only if $\mathbf{v-u} \in \{ \pm\mathbf{e}_j: 1 \leqslant j \leqslant n\}$.
We call $\mathcal{G}$ a \emph{grid}.

For $A \subset V$, the \emph{edge-boundary} of $A$ is defined to be $$\partial_{\mathcal{G}}(A):=\{ \{\mathbf{v}, \mathbf{u}\} \in E : \mathbf{u} \in A , \mathbf{v} \in V \setminus A  \}.$$
We have the following edge-isoperimetric inequality in the grid $\mathcal{G}_{N,n}$.

\begin{lemma} [Theorem 3 in \cite{Edge-isoperimetric_inequalities_in_the_grid}] \label{Lemma_edge_isoperimetric_inequality_in_grid}
    Let $\mathcal{G}=\mathcal{G}_{N,n}$ be as above.
    If $A \subset V=[[N]]^n$ such that $|A| \leqslant N^n/2$, then we have $$|\partial_{\mathcal{G}}(A)| \geqslant \min \{ |A|^{1-1/r} r N^{-1+n/r}: r=1,2, \cdots, n\} \geqslant C_n |A|^{1-1/n},$$
where $C_n := \min \{ r \cdot 2^{1/r-1/n}: r=1,2, \cdots , n\}>0$ is a constant depending only on $n$.
\end{lemma}

 \subsection{Random counting and $\varpi_{\alpha}^{(k)}$-sum over edges}

We begin with the generalization of edge-boundaries.
Let $\mathcal{G}=\mathcal{G}_{N,n}$ be the grid and $\mathscr{U}$ be a finite partition of the vertex set $V$.
An edge $\{ \mathbf{u,v}\} \in E$ is called a \emph{bridge} w.r.t $\mathscr{U}$ if $\mathbf{u,v}$ lie in different atoms of $\mathscr{U}$.
Denote by $\mathrm{Bri}(\mathcal{G}, \mathscr{U})$ the set of all bridges w.r.t. $\mathscr{U}$.
Note that $\partial_{\mathcal{G}}(A)=\mathrm{Bri}(\mathcal{G}, \mathscr{U}_A)$ where $\mathscr{U}_A:=\{A, V \setminus A\}$.
Moreover, we have
$$ 2 |\mathrm{Bri}(\mathcal{G}, \mathscr{U})|  = \sum_{U \in \mathscr{U}} |\partial_{\mathcal{G}}(U)|.$$ 

Now we let randomness come into play.
Let $m$ be a positive integer.
Given an integer $K \geqslant 2$ and $\mathbf{v} \in [[K]]^m$, we can define a partition of $\mathbb{Z}^m$ as 
$$\mathscr{Q}_{K, \mathbf{v}}:=\{\mathbf{v}+ K\mathbf{u} + [[K]]^m: \mathbf{u} \in \mathbb{Z}^m\}.$$
For each $K \geqslant 2$, we define the probability space $(\mathfrak{X}_K, \mathbb{P}_K)$ where $$\mathfrak{X}_K=\{ \mathscr{Q}_{K, \mathbf{v}} : \mathbf{v} \in [[K]]^m \}$$ and $$\mathbb{P}_K(\mathscr{Q}_{K, \mathbf{v}})=K^{-m}, \quad \forall \mathbf{v} \in [[K]]^m .$$

We begin with a simple fact.

\begin{lemma} \label{Lemma_probability_lie_in_different_atoms}
    For any $\mathbf{w} \in \mathbb{Z}^m$, we have 
    $$\mathbb{P}_K(\mathbf{0,w} \text{ lie in different atoms in }\mathscr{Q}_{K, \mathbf{v}}) \leqslant \min \left\{ \frac{|\mathbf{w}|}{K}, 1 \right\}.$$
\end{lemma}

\begin{proof}
    Write $\mathbf{w}=(w_1,w_2, \cdots, w_m)$ and $|\mathbf{w}|_{\infty}=\max_{1 \leqslant i \leqslant m}|w_i|$.
    When $|\mathbf{w}|_{\infty} \geqslant K$, we have $|\mathbf{w}|\geqslant |\mathbf{w}|_{\infty} \geqslant K$ and $$\mathbb{P}_K(\mathbf{0,w} \text{ lie in different atoms in }\mathscr{Q}_{K, \mathbf{v}})=1.$$
    When $|\mathbf{w}|_{\infty} < K$, we have 
\begin{equation*}
\begin{split}
 \mathbb{P}_K(\mathbf{0,w} \text{ lie in different atoms in }\mathscr{Q}_{K, \mathbf{v}}) &= 1- \prod_{i=1}^m \left(1-\frac{|w_i|}{K}\right)\\
  &\leqslant  \sum_{i=1}^m \frac{|w_i|}{K}=\frac{|\mathbf{w}|}{K}.
\end{split}
\end{equation*}
So the assertion follows.
\end{proof}

Now recall that $\mathcal{G}=\mathcal{G}_{N,n}=(V_{N,n},E_{N,n})$ is the grid in $\mathbb{Z}^n$ with $V=[[N]]^n$.
Given a map $f: V_{N,n} \to \mathbb{Z}^m$ and a random partition $\mathscr{Q}_{K, \mathbf{v}}$, we can consider the pullback $$f^{-1} \mathscr{Q}_{K, \mathbf{v}}:= \{f^{-1}\left(\mathbf{v}+ K\mathbf{u} + [[K]]^m \right): \mathbf{u} \in \mathbb{Z}^m\}$$
which is a finite partition of $V_{N,n}$ if we ignore empty sets (and we will always do so).

We will denote by $\mathbb{E}_K$ the expectation w.r.t. $\mathbb{P}_K$.
By double-counting on $\mathbb{E}_K \left(\left|\mathrm{Bri}\left(\mathcal{G}, f^{-1}\mathscr{Q}_{K, \mathbf{v}} \right)\right| \right)$, we can obtain the following result.

\begin{lemma} \label{Lemma_double-counting_of_expectation_of_number_of_bridges}
    With notation as above, if  $|f(V_{N,n})| \geqslant 2 K^m$, then we have
    $$ \frac{1}{4}C_n K^{-m/n} \left|f(V_{N,n}) \right| \leqslant \sum_{\{\mathbf{x,y}\} \in E_{N,n}} \min\left\{ \frac{|f(\mathbf{x})-f(\mathbf{y})|}{K}, 1\right\},$$
    where $C_n$ is the constant given in Lemma \ref{Lemma_edge_isoperimetric_inequality_in_grid}.
\end{lemma}

\begin{proof}
    On the one hand, for any fixed partition $\mathscr{Q}_{K, \mathbf{v}} \in \mathfrak{X}_K$, we see
    \begin{align*}
        2 \left|\mathrm{Bri}\left(\mathcal{G}, f^{-1}\mathscr{Q}_{K, \mathbf{v}} \right)\right|  &= \sum_{U \in f^{-1}\mathscr{Q}_{K, \mathbf{v}}} |\partial_{\mathcal{G}}(U)| \geqslant \sum_{\substack{U \in f^{-1}\mathscr{Q}_{K, \mathbf{v}};\\ |U| \leqslant N^n/2}} |\partial_{\mathcal{G}}(U)| \\
       \text{(Lemma \ref{Lemma_edge_isoperimetric_inequality_in_grid}) }  &\geqslant \sum_{\substack{U \in f^{-1}\mathscr{Q}_{K, \mathbf{v}};\\ |U| \leqslant N^n/2}} C_n \cdot |U|^{1-1/n} \\
       &\geqslant \sum_{\substack{U \in f^{-1}\mathscr{Q}_{K, \mathbf{v}};\\ |U| \leqslant N^n/2}} C_n \cdot |f(U)|^{1-1/n}.
    \end{align*}
    Note that for each $U \in f^{-1}\mathscr{Q}_{K, \mathbf{v}}$, by definition, we have $U=f^{-1}\left(\mathbf{v}+ K\mathbf{u} + [[K]]^m \right)$ for some $\mathbf{u} \in \mathbb{Z}^m $, and hence we have $$|f(U)| \leqslant |\mathbf{v}+ K\mathbf{u} + [[K]]^m|=K^m.$$
    So we obtain 
    \begin{equation}\label{Equation_upper_bound_for_bridge_in_the_proof}
        2 \left|\mathrm{Bri}\left(\mathcal{G}, f^{-1}\mathscr{Q}_{K, \mathbf{v}} \right)\right| \geqslant C_n K^{-m/n} \sum_{\substack{U \in f^{-1}\mathscr{Q}_{K, \mathbf{v}};\\ |U| \leqslant N^n/2}} |f(U)|. 
    \end{equation}

    Since $f^{-1}\mathscr{Q}_{K, \mathbf{v}}$ is a partition of $V_{N,n}=[[N]]^n$, there is at most one $U^* \in f^{-1}\mathscr{Q}_{K, \mathbf{v}}$ such that $|U^*|>N^n/2$.
    If there is no such $U^*$, (\ref{Equation_upper_bound_for_bridge_in_the_proof}) becomes $$ 2 \left|\mathrm{Bri}\left(\mathcal{G}, f^{-1}\mathscr{Q}_{K, \mathbf{v}} \right)\right| \geqslant C_n K^{-m/n} \sum_{U \in f^{-1}\mathscr{Q}_{K, \mathbf{v}}} |f(U)|=  C_n K^{-m/n} |f(V_{N,n})|. $$
    If such $U^*$ exists, we see $|f(U^*)| \leqslant K^m \leqslant |f(V_{N,n})|/2$, and hence (\ref{Equation_upper_bound_for_bridge_in_the_proof}) yields  
    \begin{equation*}
        \begin{split}
            2 \left|\mathrm{Bri}\left(\mathcal{G}, f^{-1}\mathscr{Q}_{K, \mathbf{v}} \right)\right| &\geqslant C_n K^{-m/n} \left( \left|f(V_{N,n}) \right|-|f(U^*)| \right) \\
            &\geqslant \frac{1}{2} C_n K^{-m/n} |f(V_{N,n})|.
        \end{split}
    \end{equation*}
    So in both cases, we have $$ \left|\mathrm{Bri}\left(\mathcal{G}, f^{-1}\mathscr{Q}_{K, \mathbf{v}} \right)\right| \geqslant \frac{1}{4} C_n K^{-m/n} |f(V_{N,n})|, \quad \forall  \mathscr{Q}_{K, \mathbf{v}} \in \mathfrak{X}_K.$$
    It follows that
      $$  \mathbb{E}_K \left(\left|\mathrm{Bri}\left(\mathcal{G}, f^{-1}\mathscr{Q}_{K, \mathbf{v}} \right)\right| \right)\geqslant \frac{1}{4} C_n K^{-m/n} |f(V_{N,n})|.$$

    On the other hand, for a fixed edge $\{\mathbf{x,y}\} \in E_{N,n}$, we see $\{\mathbf{x,y}\} \in \mathrm{Bri}\left(\mathcal{G}, f^{-1}\mathscr{Q}_{K, \mathbf{v}} \right)$ if and only if $ f(\mathbf{x}), f(\mathbf{y}) $ lie in different atoms of $\mathscr{Q}_{K, \mathbf{v}}$.
    Therefore, we obtain
    \begin{align*}
        \mathbb{E}_K \left(\left|\mathrm{Bri}\left(\mathcal{G}, f^{-1}\mathscr{Q}_{K, \mathbf{v}} \right)\right| \right) &= \sum_{\{\mathbf{x,y}\} \in E_{N,n}} \mathbb{P}_K(f(\mathbf{x}), f(\mathbf{y}) \text{ lie in different atoms in }\mathscr{Q}_{K, \mathbf{v}})\\
        &=\sum_{\{\mathbf{x,y}\} \in E_{N,n}} \mathbb{P}_K(0, f(\mathbf{x})-f(\mathbf{y}) \text{ lie in different atoms in }\mathscr{Q}_{K, \mathbf{v}})\\
      \text{(Lemma \ref{Lemma_probability_lie_in_different_atoms}) }  
      & \leqslant \sum_{\{\mathbf{x,y}\} \in E_{N,n}} \min \left\{ \frac{|f(\mathbf{x})-f(\mathbf{y})|}{K}, 1 \right\}.
    \end{align*}

    So we are done by comparing the upper and lower bounds for $\mathbb{E}_K \left(\left|\mathrm{Bri}\left(\mathcal{G}, f^{-1}\mathscr{Q}_{K, \mathbf{v}} \right)\right| \right)$.
\end{proof}

\begin{lemma} \label{Lemma_Estimating_sum_over_edges}
    Let $n,m,k$ be positive integers such that $n>m$.
    There are two constants $c_{n,m,k}>0$ and $M_{n,m,k}>0$ (depending only on $n,m,k$) such that for every integer $N \geqslant 1$ and every map $f:V_{N,n} \to \mathbb{Z}^m$ with $|f(V_{N,n})| \geqslant M_{n,m,k}$, we have
    \begin{equation} \label{Equation_general_estimate_for_injective_maps}
        \sum_{\{\mathbf{x,y}\} \in E_{N,n}} \varpi^{(k)}_{m/n} \left( |f(\mathbf{x})-f(\mathbf{y})| \right) \geqslant c_{n,m,k} \cdot |f(V_{N,n})| \cdot \log^{(k+1)} {|f(V_{N,n})|}.
    \end{equation}
\end{lemma}

\begin{proof}
    Note that $0<m/n<1$ since $n >m$.
    Let $j_{m/n,k}$ be the positive integer constant given in Lemma \ref{Lemma_elementary_lemma_lower_bound_of_varpi_p} (with $\alpha=m/n$).
    We first take $M'_{n,m,k}>0$ to be sufficiently large such that 
    $$J_*:= \left \lfloor \frac{1}{m} \left(-1+ \log_2{M'_{n,m,k}} \right) \right\rfloor \geqslant j_{m/n,k}$$
    Suppose that $|f(V_{N,n})| \geqslant M'_{n,m,k}$.
    In view of Lemma \ref{Lemma_elementary_lemma_lower_bound_of_varpi_p}, we have \begin{align*}
        \sum_{\{\mathbf{x,y}\} \in E_{N,n}} \varpi^{(k)}_{m/n} \left( |f(\mathbf{x})-f(\mathbf{y})| \right) 
        & \geqslant \sum_{\{\mathbf{x,y}\} \in E_{N,n}} A_{m/n,k} \sum_{j=j_{m/n,k}}^{\infty}  \frac{2^{jm/n}}{\Lambda_k(2^{j})} \cdot \min\left\{\frac{|f(\mathbf{x})-f(\mathbf{y})|}{2^j}, 1\right\}\\
        & \geqslant A_{m/n,k} \sum_{j=j_{m/n,k}}^{J_f} \left( \frac{2^{jm/n}}{\Lambda_k(2^{j})} \cdot \sum_{\{\mathbf{x,y}\} \in E_{N,n}}  \min\left\{\frac{|f(\mathbf{x})-f(\mathbf{y})|}{2^j}, 1\right\} \right).
    \end{align*}
    where $A_{m/n,k}>0$ is the constant given in Lemma \ref{Lemma_elementary_lemma_lower_bound_of_varpi_p} and $$J_f:= \left \lfloor \frac{1}{m} \left(-1+ \log_2{|f(V_{N,n})|} \right) \right\rfloor \geqslant J_* \geqslant j_{m/n,k}.$$
    Note that for all $j_{m/n,k} \leqslant j \leqslant J_f$, we have $ 2 \cdot 2^{jm} \leqslant 2 \cdot  2^{J_f m} \leqslant |f(V_{N,n})|$.
    By applying Lemma \ref{Lemma_double-counting_of_expectation_of_number_of_bridges} to the above inequality with $K=2^{j}$ ($j_{m/n,k} \leqslant j \leqslant J_f$), we see
    \begin{align} \label{Equation_1}
        \sum_{\{\mathbf{x,y}\} \in E_{N,n}} \varpi^{(k)}_{m/n} \left( |f(\mathbf{x})-f(\mathbf{y})| \right)
        & \geqslant A_{m/n,k} \sum_{j=j_{m/n,k}}^{J_f} \left( \frac{2^{jm/n}}{\Lambda_k(2^{j})} \cdot \frac{1}{4}C_n  2^{-jm/n} |f(V_{N,n})| \right)\\
      \notag  & = \frac{1}{4}C_n A_{m/n,k} \cdot |f(V_{N,n})| \cdot \sum_{j=j_{m/n,k}}^{J_f}  \frac{1}{\Lambda_k(2^{j})}. 
    \end{align}
    So it remains to estimate $\sum_{j=j_{m/n,k}}^{J_f}  \frac{1}{\Lambda_k(2^{j})}$.
    We first note that for each $J_f \geqslant J_*$ and $j_{m/n,k} \leqslant j \leqslant J_f$, we have $$\log^{(i)}(2^j)= \log^{(i-1)}(\log2 \cdot j) \leqslant \log^{(i-1)}(j), \quad \forall 1 \leqslant i \leqslant k,$$
    with the convention $\log^{(0)}(t)=t$.
    Hence we have $$\Lambda_k(2^j) = \prod_{i=1}^k \log^{(i)}(2^j) \leqslant \prod_{i=1}^k \log^{(i-1)}(j)= j \cdot \Lambda_{k-1}(j),$$
    with the convention $\Lambda_0(t)=1$.
    So for each $J_f \geqslant J_*$ we have 
    \begin{equation} \label{Equation_2}
        \begin{split}
        \sum_{j=j_{m/n,k}}^{J_f}  \frac{1}{\Lambda_k(2^{j})} & \geqslant \sum_{j=j_{m/n,k}}^{J_f} \frac{1}{j\cdot \Lambda_{k-1}(j)} \\
        &\geqslant \int_{j_{m/n,k}}^{J_f+1}\frac{1}{t \cdot \Lambda_{k-1}(t)} \mathrm{d}t \\
        &= \log^{(k)}(J_f+1)- \log^{(k)} (j_{m/n,k}).
        \end{split}
    \end{equation}
    Note that if $\left|f(V_{N,n})\right|$ is sufficiently large, we have 
    \begin{equation} \label{Equation_3}
    \log^{(k)}(J_f+1) \geqslant \log^{(k)} {\left(  \frac{1}{m} \left(-1+ \log_2{|f(V_{N,n})|} \right) \right) } \geqslant \frac{1}{2} \log^{(k+1)} {\left| f(V_{N,n}) \right|}.
    \end{equation}

    We are done in view of (\ref{Equation_1}), (\ref{Equation_2}), and (\ref{Equation_3}).
\end{proof}

\section{Proof of Theorem \ref{Theorem_main_result_measure_equivalence_version}: absence of $\varpi^{(k)}_{m/n}$-integrability}

This section aims to complete the proof of Theorem \ref{Theorem_main_result_measure_equivalence_version}.
To begin with, let $\mathcal{C}=(X,X_G,X_H, \mu )$ be a measure equivalence coupling from $G$ to $H$, where $G,H$ are two countable (and finitely generated) groups.
After scaling, we will assume $\mu(X_H)=1$.
Recall that $\kappa^{\mathcal{C}}: G \times X_{H} \to H$ is the cocycle with respect to $\mathcal{C}$.
For $x \in X_{H}$, we define $\kappa_{x}^{\mathcal{C}}: G \to H$ by $\kappa_x^{\mathcal{C}} (g):=\kappa^{\mathcal{C}}(g,x)$.

Roughly speaking, Lemma \ref{Lemma_Estimating_sum_over_edges} says that the $\varpi^{(k)}_{\alpha}$-sum (over edges) for a function with large image is always large.
The next lemma says that for $x$ in a subset of $X_H$ of positive measure, the map $\kappa^{\mathcal{C}}_x$ has large image.

\begin{lemma} \label{Lemma_many_x_has_large_image}
    Let $\mathcal{C}=(X,X_G,X_H, \mu )$ be a measure equivalence coupling from $G$ to $H$.
    There is a constant $0<\rho_{\mathcal{C}}<1$ depending only on the coupling $\mathcal{C}$ such that for every finite set $F \subset G$, we have
        $$\mu \left( \left\{x \in X_H:  \left| \kappa^{\mathcal{C}}_x (F) \right| \geqslant \rho_{\mathcal{C}} |F| \right\}\right) \geqslant \rho_{\mathcal{C}}.$$
\end{lemma}

\begin{proof}
    For an integer $r>0$, we will denote by $B_G(r)$ the ball of radius $r$ (with respect to word-length metric on $G$), centered at the identity element $e_G$.
    Since $X_G$ is a fundamental domain, we see $$X= \bigcup_{r=1}^{\infty} B_G(r) X_G, \quad \text{(modulo a null set)} $$
     where $B_G(r) X_G := \{ gx: g \in B_G(r), x \in X_G\}$.
    So there is $r_0 \geqslant 1$ such that $$\mu (X_H \cap B_{G}(r_0)X_G)> 0.$$
    We will write $X_0:=X_H \cap B_{G}(r_0)X_G$.
    For each $x \in X_H$, the hitting time set of $x$ to $X_0$ under the induced action $\star$ by $G$ is defined to be $$R_{X_0, \star} (x):=\{ g \in G : g \star x \in X_0\} \subset G.$$

    We first claim that for $\mu|_{X_H}$-a.e. $x \in X_H$, the map $\kappa^{\mathcal{C}}_x$ is at most $|B_G(r_0)|$-to-one on the set $R_{X_0, \star}(x)$; i.e., for $\mu|_{X_H}$-a.e. $x \in X_H$ and each $h \in H$, we have 
    \begin{equation} \label{Equation_finite_to_one_on_the_return_time_set}
        \left| \left( \kappa^{\mathcal{C}}_{x}\right)^{-1}(h) \cap R_{X_0, \star} (x) \right| \leqslant |B_{G}(r_0)|.
    \end{equation}
    To see the above inequality, we take $g \in \left( \kappa^{\mathcal{C}}_{x}\right)^{-1}(h) \cap R_{X_0, \star} (x)$ for any $h \in H$ and a \emph{good} $x \in X_H$ (to be specified later).
    By definition, we have $hgx= g \star x \in X_0= X_H \cap B_G(r_0) X_G $ since $g \in R_{X_0, \star}(x)$.
    So there is $\hat{g} \in B_G(r_0)$ such that $\hat{g}^{-1}hgx \in X_G$; i.e., $(\hat{g}^{-1}g)(hx) \in X_G$.
    Since $X_G$ is a fundamental domain, we see $$\{(\hat{g}^{-1}g)(hx)\}= G(hx) \cap X_G.$$
    Note that $G$-action on $X$ is free, which implies that $\hat{g}= g\beta^{-1}$, where $\beta:=\beta(h,x)$ is the unique element in $G$ such that $\beta hx \in X_G$.
    This means that $g \mapsto \hat{g}$ is an injection from $\left( \kappa^{\mathcal{C}}_{x}\right)^{-1}(h) \cap R_{X_0, \star} (x)$ to $B_{G}(r_0)$.
    So we obtain the inequality (\ref{Equation_finite_to_one_on_the_return_time_set}).
    We now explain \emph{goodness} for $x \in X_H$.
    Indeed, we say $x \in X_H$ is good if for all $h \in H$, the map $g \mapsto gx$ is a bijection from $G$ to $Gx$ and $|X_G \cap G(hx)|=1$.
    Such good points form a subset of $X_H$ of full measure since the action $G \curvearrowright X$ is free and $X_G$ is a fundamental domain of this action.
    So (\ref{Equation_finite_to_one_on_the_return_time_set}) holds for all $h \in H$ and $\mu|_{X_H}$-a.e. $x \in X_H$.

    Now for any non-empty finite set $F \subset G$ and $\mu |_{X_H}$-a.e. $x \in X_H$, by (\ref{Equation_finite_to_one_on_the_return_time_set}) we have 
    \begin{align*}
        |F \cap R_{X_0,\star}(x)| & \leqslant \sum_{h \in \kappa^{\mathcal{C}}_x(F)} \left| \left( \kappa^{\mathcal{C}}_{x}\right)^{-1}(h) \cap R_{X_0, \star} (x) \right| \\
         & \leqslant \sum_{h \in \kappa^{\mathcal{C}}_x(F)} |B_{G}(r_0)| = |B_{G}(r_0)| \cdot \left| \kappa^{\mathcal{C}}_x(F) \right|.
    \end{align*}
    Consequently, we have 
    \begin{align*}
        \int_{X_H} \left|\kappa^{\mathcal{C}}_x (F)\right| \mathrm{d}\mu(x) 
        &  \geqslant \frac{1}{|B_{G}(r_0)|} \int_{X_H} \left|F \cap R_{X_0,\star}(x)\right|  \mathrm{d}\mu(x) \\
        & =\frac{1}{|B_{G}(r_0)|} \sum_{g \in F} \mu \left( \left\{  x \in X_H: g \star x \in X_0   \right\}\right)\\
        & =\frac{1}{|B_{G}(r_0)|} \sum_{g \in F} \mu( g^{-1} \star X_0) \\
       \text{(since the action $\star$ preserves $\mu$) } &   = \frac{\mu(X_0)}{|B_{G}(r_0)|}\cdot |F|.
    \end{align*}
    Recall that $0< \mu(X_0) \leqslant \mu(X_H)=1$.
    So we can take $$\rho_{\mathcal{C}}:= \frac{ \mu(X_0)}{ 2 |B_{G}(r_0)|} \in (0,1).$$
    Namely, we have $\int_{X_H} \left|\kappa^{\mathcal{C}}_x (F)\right| \mathrm{d}\mu(x) \geqslant 2 \rho_{\mathcal{C}} |F| $.
    Consequently, we see
    \begin{align*}
        \mu \left( \left\{x \in X_H:  \left| \kappa^{\mathcal{C}}_x (F) \right| \geqslant \rho_{\mathcal{C}} |F| \right\} \right)
        & \geqslant  \int_{\left\{x \in X_H:  \left| \kappa^{\mathcal{C}}_x (F) \right| \geqslant \rho_{\mathcal{C}} |F| \right\}} \frac{\left| \kappa^{\mathcal{C}}_x (F) \right|}{|F|} \mathrm{d}\mu \\
        & =  \int_{X_H} \frac{\left|\kappa^{\mathcal{C}}_x (F)\right|}{|F|} \mathrm{d}\mu(x)- \int_{\left\{x \in X_H:  \left| \kappa^{\mathcal{C}}_x (F) \right| < \rho_{\mathcal{C}} |F| \right\}} \frac{\left| \kappa^{\mathcal{C}}_x (F) \right|}{|F|} \mathrm{d}\mu \\
        & \geqslant 2 \rho_{\mathcal{C}}- \int_{\left\{x \in X_H:  \left| \kappa^{\mathcal{C}}_x (F) \right| < \rho_{\mathcal{C}} |F| \right\}}  \rho_{\mathcal{C}} \mathrm{d}\mu \\
        & \geqslant \rho_{\mathcal{C}}.
    \end{align*}
This completes the proof (since when $F= \varnothing$
 the assertion is trivial).
\end{proof}

Now we are ready to prove the main result.

\begin{proof}[Proof of Theorem \ref{Theorem_main_result_measure_equivalence_version}]
We assume $k \geqslant 1$ since the assertion for larger $k$ implies the assertion for smaller $k$.
Let $n>m$ be two positive integers and $\mathcal{C}=(X, X_{\mathbb{Z}^n}, X_{\mathbb{Z}^m}, \mu)$ be a measure equivalence from $\mathbb{Z}^n$ to $\mathbb{Z}^m$.
Our aim is to show $\mathcal{C}$ is not $\varpi^{(k)}_{m/n}$-integrable.
As above, we assume $\mu(X_{\mathbb{Z}^m})=1$ (otherwise, scale $\mu$ by a constant).

By Lemma \ref{Lemma_many_x_has_large_image}, we see there is a constant $\rho_{\mathcal{C}} \in (0,1)$ such that for each integer $N \geqslant 1$ there is a subset $Y_N \subset X_{\mathbb{Z}^m}$ satisfying $\mu(Y_N) \geqslant \rho_{\mathcal{C}}$ and 
\begin{equation} \label{Equation_largeness_of_image_in_the_proof_of_the_main_result}
\left| \kappa^{\mathcal{C}}_x (V_{N,n})  \right| \geqslant \rho_{\mathcal{C}} |V_{N,n}| =  \rho_{\mathcal{C}} N^n, \quad \forall x \in Y_N.    
\end{equation}

Now let $c_{n,m,k}>0$ and $M_{n,m,k}>0$ be the constants given in Lemma \ref{Lemma_Estimating_sum_over_edges}.
Now for sufficiently large $N$ (such that $\rho_{\mathcal{C}} N^n \geqslant M_{n,m,k}$), we see $\left| \kappa^{\mathcal{C}}_x (V_{N,n})  \right| \geqslant  M_{n,m,k}$.
By the inequality (\ref{Equation_general_estimate_for_injective_maps}), for all $x \in Y_N$, we have
\begin{equation} \label{Equation_estimate_for_map_induced_by_cocycle}
        \sum_{\{\mathbf{u,v}\} \in E_{N,n}} \varpi^{(k)}_{m/n} \left( \left|\kappa^{\mathcal{C}}_x(\mathbf{u})-\kappa^{\mathcal{C}}_x(\mathbf{v}) \right| \right) \geqslant c_{n,m,k} \cdot \left|\kappa^{\mathcal{C}}_x(V_{N,n}) \right| \cdot \log^{(k+1)} {\left|\kappa^{\mathcal{C}}_x(V_{N,n})\right|}.
    \end{equation}
Since $\{\mathbf{u,v}\} \in E_{N,n}$, in view of the cocycle identity (\ref{Equation_cocycle_identity_for_measure_equivalence_cocyle}), we have 
$$\kappa^{\mathcal{C}}_x(\mathbf{u})-\kappa^{\mathcal{C}}_x(\mathbf{v}) = \kappa^{\mathcal{C}}_{\mathbf{v} \star x}(\mathbf{\mathbf{u-v}}) = \kappa^{\mathcal{C}}_{\mathbf{v} \star x}(\mathbf{e}_j),$$
for some $1 \leqslant j \leqslant n$.
So the inequality (\ref{Equation_estimate_for_map_induced_by_cocycle}) is nothing but, for all $x \in Y_N$, $$\sum_{j=1}^n \sum_{  \substack{ \mathbf{v} \in V_{N,n}; \\ \mathbf{v}+\mathbf{e}_j \in V_{N,n}}} \varpi^{(k)}_{m/n} \left( \left|\kappa^{\mathcal{C}}_{\mathbf{v} \star x}(\mathbf{e}_j) \right| \right) \geqslant c_{n,m,k} \cdot \left|\kappa^{\mathcal{C}}_x(V_{N,n}) \right| \cdot \log^{(k+1)}{\left|\kappa^{\mathcal{C}}_x(V_{N,n})\right|}. $$
This together with (\ref{Equation_largeness_of_image_in_the_proof_of_the_main_result}) implies that for all sufficiently large $N$ and for all $x \in Y_N$, we have
$$\sum_{j=1}^n \sum_{  \substack{ \mathbf{v} \in V_{N,n}; \\ \mathbf{v}+\mathbf{e}_j \in V_{N,n}}} \varpi^{(k)}_{m/n} \left( \left|\kappa^{\mathcal{C}}_{\mathbf{v} \star x}(\mathbf{e}_j) \right| \right) \geqslant c_{n,m,k}  \rho_{\mathcal{C}} \cdot N^n \log^{(k+1)} {(\rho_{\mathcal{C}} N^n)}. $$
Taking integration in the above yields
\begin{equation} \label{Equation_almost_done}
    \sum_{j=1}^n \sum_{  \substack{ \mathbf{v} \in V_{N,n}; \\ \mathbf{v}+\mathbf{e}_j \in V_{N,n}}} \int_{X_{\mathbb{Z}^m}} \varpi^{(k)}_{m/n} \left( \left|\kappa^{\mathcal{C}}_{\mathbf{v} \star x}(\mathbf{e}_j) \right| \right) \mathrm{d}\mu(x) 
    \geqslant  c_{n,m,k} \rho_{\mathcal{C}}^2 \cdot  N^n \log^{(k+1)} {(\rho_{\mathcal{C}} N^n)}.
\end{equation}

Note that the induced action $\star$ preserves measure $\mu|_{X_{\mathbb{Z}^m}}$ and hence we have $$\int_{X_{\mathbb{Z}^m}} \varpi^{(k)}_{m/n} \left( \left|\kappa^{\mathcal{C}}_{\mathbf{v} \star x}(\mathbf{e}_j) \right| \right) \mathrm{d}\mu(x)=\int_{X_{\mathbb{Z}^m}} \varpi^{(k)}_{m/n} \left( \left|\kappa^{\mathcal{C}}_{x}(\mathbf{e}_j) \right| \right) \mathrm{d}\mu(x), \quad \forall \mathbf{v} \in V_{N,n}.$$
Moreover, for each fixed $1 \leqslant j \leqslant n$, we have $$ \sum_{  \substack{  \mathbf{v} \in V_{N,n}, \\ \mathbf{v}+\mathbf{e}_j \in V_{N,n}}} 1  \leqslant \sum_{\mathbf{v} \in V_{N,n}} 1 = |V_{N,n}|= N^n.$$
It follows by (\ref{Equation_almost_done}) that, for all sufficiently large $N$, we have
$$\sum_{j=1}^n \int_{X_{\mathbb{Z}^m}} \varpi^{(k)}_{m/n} \left( \left|\kappa^{\mathcal{C}}_{x}(\mathbf{e}_j) \right| \right) \mathrm{d}\mu(x) 
    \geqslant c_{n,m,k} \rho_{\mathcal{C}}^2 \cdot \log^{(k+1)} {(\rho_{\mathcal{C}} N^n)}.$$
Taking $N \to +\infty$ yields $$\sum_{j=1}^n \int_{X_{\mathbb{Z}^m}} \varpi^{(k)}_{m/n} \left( \left|\kappa^{\mathcal{C}}_{x}(\mathbf{e}_j) \right| \right) \mathrm{d}\mu(x)=\infty. $$
Hence there is $1 \leqslant j_0 \leqslant n$ such that $$\int_{X_{\mathbb{Z}^m}} \varpi^{(k)}_{m/n} \left( \left|\kappa^{\mathcal{C}}(\mathbf{e}_{j_0},x) \right| \right) \mathrm{d}\mu(x)=\int_{X_{\mathbb{Z}^m}} \varpi^{(k)}_{m/n} \left( \left|\kappa^{\mathcal{C}}_{x}(\mathbf{e}_{j_0}) \right| \right) \mathrm{d}\mu(x)=\infty.$$
This means that $\mathcal{C}$ is not $\varpi^{(k)}_{m/n}$-integrable, in view of Proposition \ref{Proposition_the_same_integrability_up_to_constant}.
\end{proof}

\section{$\left(\varpi^{(k)}_{m/n, \epsilon}, \varpi^{(k)}_{n/m, \epsilon}\right)$-orbit equivalence for $p$-adic adding machines}

\subsection{$p$-Adic numbers and adding machines}
Here we briefly introduce some basic fact about $p$-adic numbers and $p$-adic adding machines.
We refer the reader to the classical monograph \cite{Applied_algebraic_dynamics} by Anashin and Khrennikov for more details.

Let $p$ be a prime number. For a nonzero integer $n \in \mathbb{Z}$, denote by $v_p(n)$ the largest nonnegative integer $k$ such that $p^k$ divides $n$.
The $p$-\emph{adic absolute value} on $\mathbb{Z}$ is defined by $|n|_p=p^{-v_p(n)}$ for $n \neq 0$ and $|0|_p = 0$.
It is easy to check that
$$|x+y|_p \leqslant \max\{|x|_p,|y|_p\}, \quad \forall x,y \in \mathbb{Z}.$$
Hence $|\cdot|_p$ induces a metric on $\mathbb{Z}$, called the $p$-\emph{adic metric}.

The \emph{ring of $p$-adic integers}, denoted by
$\mathbb{Z}_p$, is the completion of $\mathbb{Z}$ under the $p$-adic metric.
Another equivalent definition in terms of inverse limit says that, as a topological ring, we have $$\mathbb{Z}_p = \varprojlim_{n} \mathbb{Z}/ p^n \mathbb{Z}.$$ 
Note that every element of $\mathbb{Z}_p$ admits a unique expansion $$z=\sum_{i \geqslant 0} a_ip^i,     \qquad a_i \in \{0,1, \dots ,p-1\}= \mathbb{Z}/p\mathbb{Z}, \  \forall i \geqslant 0.$$
This naturally defines, for each $i \geqslant 0$, the ($p$-\emph{adic}) \emph{$i$-th digit map} $\delta_i : \mathbb{Z}_p \to \mathbb{Z}/ p \mathbb{Z}$ with $\delta_i(z)= a_i$.
Note that $z \in \mathbb{Z}_p$ is a non-negative (rational) integer if and only if $z$ is finitely supported; i.e., $\delta_i(z) \neq 0$ for only finitely many $i \geqslant 0$; and $z \in \mathbb{Z}_p$ is a negative (rational) integer if and only if $\delta_i(z) \neq p-1$ for only finitely many $i \geqslant 0$.
We will denote by $\mathcal{Z}_p : =\mathbb{Z}_p \setminus \mathbb{Z}$ the $p$-adic integers which are not (rational) integers.

Now let $k \geqslant l \geqslant 0$ and we put $\delta_{[l,k]} : \mathbb{Z}_p \to (\mathbb{Z} / p \mathbb{Z})^{k-l+1}$ with $\delta_{[l,k]}(z)=(\delta_{l}(z),\delta_{l+1}(z), \cdots, \delta_{k}(z))$.
The \emph{Haar measure} on $\mathbb{Z}_p$, denoted by $\mu_p$, is uniquely determined by its values on cylinder sets: $$\mu_p ( \delta_{[l,k]}^{-1} (\overline{a}) ):= p^{-(k-l+1)}, \quad \forall \overline{a} \in (\mathbb{Z} / p \mathbb{Z})^{k-l+1}.$$
Note that $\mu_p$ is atom-free, which gives $\mu_p(\mathcal{Z}_p)=1$.

For $d \geqslant 1$, the product metric on $\mathbb{Z}_p^d$ is defined by $$\mathrm{dist}(\mathbf{z}, \mathbf{w})= \sum_{i=1}^{d} \mathrm{dist}(z_i, w_i)= \sum_{i=1}^{d} |z_i-w_i|_p,$$
where $\mathbf{z}=(z_1, z_2, \cdots, z_d), \mathbf{w}=(w_1, w_2, \cdots, w_d) \in \mathbb{Z}_p^d$.
We will denote by $\mu^{d}_p$ the Haar measure on $\mathbb{Z}_p^d$, which is the product measure for $d$ copies of $\mu_p$.
We now introduce a bijection between $\mathbb{Z}_p^d$ and $\mathbb{Z}_p$ which will be used later to construct orbit equivalences with prescribed quantitative restrictions.
For $\mathbf{z}=(z_1, z_2, \cdots,z_d) \in \mathbb{Z}_p^d$, we let $B_d(\mathbf{z}) \in \mathbb{Z}_p$ such that $$\delta_{jd+r}(B_d(\mathbf{z}))= \delta_j(z_{r+1}), \quad \forall j \geqslant 0, \ r \in \{0,1, \cdots, d-1\}.$$
Intuitively, $B_d$ reads the digits of $z_1, z_2, \cdots,z_d$ in order.
It is easy to see that $B_d : \mathbb{Z}_p^d \to \mathbb{Z}_p$ is a homeomorphism and a measure isomorphism (w.r.t. the Haar measures). 
See, for example, \cite{Applied_algebraic_dynamics} (pp. 124-125) for more details about the bijection $B_d$.

The ($d$\emph{-dimensional}) $p$-\emph{adic adding machine} is the action $\mathbb{Z}^d \curvearrowright \mathbb{Z}_p^d$ with addition action $\mathbf{n+z}$ for each $\mathbf{n} \in \mathbb{Z}^d$ and $\mathbf{z} \in \mathbb{Z}_p^d$.
Note that this action preserves $\mu^{d}_p$ and is free (w.r.t. $\mu^{d}_p$).
Moreover, it is (topologically conjugate to) a $\mathbb{Z}^d$-odometer.

Two points $z,w \in \mathbb{Z}_p$ are said to be \emph{cofinal} if $\delta_i(z) \neq \delta_i(w)$ for only finitely many $i \geqslant 0$.
We say $\mathbf{z,w} \in \mathbb{Z}_p^d$ are \emph{cofinal} if all the pairs of the corresponding coordinates are cofinal.
We list two basic facts here which, we believe, can be easily verified by the reader.

\begin{lemma} \label{Lemma_cofinal_and_orbit_partition_in_p_adic_adding_machine}
    In the $d$-dimensional $p$-adic adding machine, $\mathcal{Z}_p^d$ is a disjoint union of orbits and $\mathbf{z,w} \in \mathcal{Z}_p^d$ lie in the same orbit if and only if they are cofinal.
\end{lemma}

\begin{lemma} \label{Lemma_B_preserves_cofinalness}
   Let $\mathbf{z,w} \in \mathbb{Z}_p^d$.
   Then $\mathbf{z,w}$ are cofinal if and only if $B_d(\mathbf{z})$ and $B_d(\mathbf{w})$ are cofinal.
\end{lemma}

\subsection{Proof of Theorem \ref{Theorem_quantitative_orbit_equivalence_for_p_adic_adding_machines}}

Let $p$ be a prime and $n,m$ be two positive integers.
Let $\mathbb{Z}^n \curvearrowright (\mathbb{Z}_p^n, \mu_p^n)$ and $\mathbb{Z}^m \curvearrowright (\mathbb{Z}_p^m, \mu_p^m)$ be two $p$-adic adding machines.
Consider $\Phi = \Phi_{n,m}:= B_{m}^{-1} \circ B_n : \mathbb{Z}_p^n \to \mathbb{Z}_p^m$, which is clearly a homeomorphism and a measure isomorphism.
Put $X:= \mathcal{Z}_p^n \cap \Phi^{-1}( \mathcal{Z}_p^m)$ and $Y:= \Phi(X)$.
We note that $\mu_p^n(X)=\mu_p^m(Y)=1$ since $\mu_p(\mathcal{Z}_p)= \mu_p(\mathbb{Z}_p \setminus \mathbb{Z})=1$.

\begin{lemma}
    With the notation above, we have $$\Phi(\mathbf{x}+\mathbb{Z}^n)=\Phi(\mathbf{x})+\mathbb{Z}^m , \quad \forall \mathbf{x} \in X.$$
    Consequently, $\Phi$ is an orbit equivalence between the two $p$-adic adding machines.
\end{lemma}

\begin{proof}
    Let $\mathbf{x} \in X \subset \mathcal{Z}_p^n$.
    On the one hand, for any $\mathbf{w} \in \mathbf{x}+\mathbb{Z}^n  $, by Lemma \ref{Lemma_cofinal_and_orbit_partition_in_p_adic_adding_machine}, $\mathbf{x,w}$ are cofinal.
    Then by Lemma \ref{Lemma_B_preserves_cofinalness}, we see $\Phi(\mathbf{x}), \Phi(\mathbf{w})$ are cofinal.
    Since $\Phi(\mathbf{x}) \in \Phi(X)=Y \subset \mathcal{Z}_p^m$, we have $\Phi(\mathbf{w}) \in \Phi(\mathbf{x})+\mathbb{Z}^m $.
    This then gives $\Phi(\mathbf{x}+ \mathbb{Z}^n) \subset \Phi(\mathbf{x})+ \mathbb{Z}^m $.
    On the other hand, for each $\mathbf{y} \in \Phi(\mathbf{x})+\mathbb{Z}^m $, we see $\mathbf{y}$ and $\Phi(\mathbf{x})$ are cofinal, so are $\Phi^{-1}(\mathbf{y})$ and $\mathbf{x}$.
    It follows that $ \Phi^{-1}(\mathbf{y}) \in \mathbf{x}+\mathbb{Z}^n$ and thus $\Phi(\mathbf{x}+\mathbb{Z}^n) \supset \Phi(\mathbf{x})+ \mathbb{Z}^m$.
\end{proof}

We then show that this orbit equivalence is as required.

\begin{proof}[Proof of Theorem \ref{Theorem_quantitative_orbit_equivalence_for_p_adic_adding_machines}]
With the notation above, for a fixed integer $1 \leqslant j \leqslant n$ and a point $\mathbf{x}=(x_1, x_2, \cdots,x_n) \in X$, we put $$\ell_j(\mathbf{x}):= \min\{k \geqslant 0: \delta_k(x_j) \neq p-1\}.$$
Note that $\ell_j: X \to \mathbb{Z}_{\geqslant 0}$ for all $1 \leqslant j \leqslant n$ since $-1 \notin \mathcal{Z}_p$.
Moreover, $\mathbf{e}_j + \mathbf{x}$ and $\mathbf{x}$ differ only in the first $\ell_j(\mathbf{x})+1$ digits at most; namely, if we write $\mathbf{e}_j+ \mathbf{x}=(x_1', x_2', \cdots,x_n')$, then $\delta_k(x_i)=\delta_k(x_i')$ for all $1 \leqslant i \leqslant n$ and $k \geqslant \ell_j(\mathbf{x})+1.$
This then gives that $\Phi(\mathbf{e}_j+\mathbf{x})$ and $\Phi(\mathbf{x})$ differ in the first $$\hat{\ell}_j(\mathbf{x}):= 1+ \left\lfloor \frac{n (1+\ell_j (\mathbf{x}))}{m}  \right \rfloor$$ digits at most, by the definition of $\Phi$.

Now let $\kappa^{\Phi}$ be the corresponding cocycle.
Then we have $ \Phi(\mathbf{e}_j+\mathbf{x}) =\kappa^{\Phi}(\mathbf{e}_j, \mathbf{x})+ \Phi(\mathbf{x}).$
Writing $\Phi(\mathbf{x})=(y_1,y_2, \cdots, y_m)$ and $ \Phi(\mathbf{e}_j+\mathbf{x})=(y_1', y_2', \cdots, y_m')$, we obtain 
\begin{equation} \label{Equation_upper_bound_of_cocycle_in_p_adic_adding_machines}
      |\kappa^{\Phi}(\mathbf{e}_j, \mathbf{x})| =  \sum_{i=1}^{m} \left| \sum_{k=0}^{\hat{\ell}_j(\mathbf{x})}p^k\left(\delta_k(y_i')-\delta_k(y_i)\right) \right|\leqslant \sum_{i=1}^{m} \sum_{k=0}^{\hat{\ell}_j(\mathbf{x})}p^k(p-1) \leqslant m \cdot p^{\hat{\ell}_j(\mathbf{x})+1},
\end{equation}
where $\delta_k(y_i')-\delta_k(y_i)$ is the subtraction in $\mathbb{Z}$.
Moreover, for each $1 \leqslant j \leqslant n$ and $l \in \mathbb{Z}_{\geqslant 0}$, we have 
\begin{equation*}
    \begin{split}
\mu_p^n \left(\ell_j ^{-1}(l)\right)  &= \mu_p \left(\delta_{[0,l-1]}^{-1}(p-1, p-1, \cdots, p-1) \setminus \delta_{[0,l]}^{-1}(p-1, p-1, \cdots, p-1)\right) \\
&= p^{-l}-p^{-(l+1)},
    \end{split}
\end{equation*}
where we put $\delta_{[0,-1]}^{-1}= \mathbb{Z}_p$.

Now for any fixed $\epsilon>0$ and any fixed integer $k \geqslant 1$, we let $L_{\epsilon}$ be the least non-negative integer such that $$  m \cdot p^{ n (1+L_{\epsilon})/m +2} \geqslant T_{m/n, k, \epsilon}.$$
Since $\varpi^{(k)}_{m/n, \epsilon}$ is increasing, for each integer $1 \leqslant j \leqslant n$, we have
\begin{align*}
    \int_{\mathbb{Z}_p^n} \varpi^{(k)}_{m/n, \epsilon} \left(|\kappa^{\Phi}(\mathbf{e}_j, \mathbf{x})| \right)\mathrm{d}\mu_p^n(\mathbf{x}) 
    & \leqslant \int_{X}  \varpi^{(k)}_{m/n, \epsilon} \left( m \cdot p^{\hat{\ell}_j(\mathbf{x})+1} \right) \mathrm{d}\mu_p^n(\mathbf{x}) \qquad  \text{ by } (\ref{Equation_upper_bound_of_cocycle_in_p_adic_adding_machines})\\
    & = \sum_{l=0}^{+\infty} \varpi^{(k)}_{m/n, \epsilon} \left( m \cdot p^{ \left\lfloor n (1+l)/m  \right \rfloor+2} \right)\cdot \mu_p^n \left(\ell_j ^{-1}(l)\right) \\
    & \leqslant L_{\epsilon} \cdot \varpi^{(k)}_{m/n, \epsilon} (0) + \sum_{l=L_{\epsilon}}^{+\infty} \varpi^{(k)}_{m/n, \epsilon} \left( m \cdot p^{ n (1+l)/m +2} \right)\cdot \mu_p^n \left(\ell_j ^{-1}(l)\right)\\
   & = L_{\epsilon} \cdot \varpi^{(k)}_{m/n, \epsilon} (0) + P \cdot \sum_{l=L_{\epsilon}}^{+\infty} \frac{  \left( \log^{(k)} \left( Q \cdot p^{nl/m} \right) \right)^{-\epsilon}}{\Lambda_k\left( Q \cdot p^{nl/m} \right)}
    \end{align*}
    where $P:=P_{n,m,p}:=(m p^2)^{m/n} \cdot (p-1)>0$ and $Q:=Q_{n,m,p}:=mp^{n/m+2}>0$.
    Note that $$ \lim_{l \to \infty} \frac{ \log^{(k)} \left( Q \cdot p^{nl/m} \right)}{ \log^{(k-1)}{(l)}} \quad  \text{and} \quad \lim_{l \to \infty} \frac{\Lambda_k\left( Q \cdot p^{nl/m} \right)}{l \cdot \Lambda_{k-1}(l)}$$
    both exist and are positive constants (depending only on $n,m,p$), with the convention $\log^{(0)}(l)=l$ and $\Lambda_{0}(l)=1$.
    Since it is easy to see (by the integral test) $$\sum_{l} \frac{\left( \log^{(k-1)}(l) \right)^{-\epsilon}}{l \cdot \Lambda_{k-1}(l) \ }< + \infty,$$
    the comparison test then gives
    $$\sum_{l=L_{\epsilon}}^{+\infty} \frac{  \left( \log^{(k)} \left( Q \cdot p^{nl/m} \right) \right)^{-\epsilon}}{\Lambda_k\left( Q \cdot p^{nl/m} \right)}< +\infty.$$
    So we see $$\int_{\mathbb{Z}_p^n} \varpi^{(k)}_{m/n, \epsilon} \left(|\kappa^{\Phi}(\mathbf{e}_j, \mathbf{x})| \right)\mathrm{d}\mu_p^n(\mathbf{x}) < + \infty,$$
    for any integer $1 \leqslant j \leqslant n$, any integer $k \geqslant 1$ and any real number $\epsilon>0$.

 In view of Remark \ref{Remark_checking_integrability_for_generating_set}, we see $\Phi$ is $\varpi^{(k)}_{m/n, \epsilon}$-integrable for every $\epsilon>0$ and every integer $k \geqslant 1$. 
 A similar argument shows that the orbit equivalence $\Phi^{-1}=B_n^{-1} \circ B_m$ is $\varpi^{(k)}_{n/m, \epsilon}$-integrable for every $\epsilon>0$ and every integer $k \geqslant 1$.
 It then follows by definition that $\mathbb{Z}^n \curvearrowright (\mathbb{Z}_p^n, \mu_p^n)$ is $\left(\varpi^{(k)}_{m/n, \epsilon}, \varpi^{(k)}_{n/m, \epsilon} \right)$-orbit equivalent to $\mathbb{Z}^m \curvearrowright (\mathbb{Z}_p^m, \mu_p^m)$ for all $\epsilon>0$.
\end{proof}

\bigskip

\noindent\textbf{Acknowledgments.} 
The author would like to warmly thank Jia-Yan Yao for interesting discussions on the subject. 
He would like also to heartily thank the National Natural Science Foundation of China (Grant No.\,12231013) for partial financial support.

\bibliographystyle{amsalpha}
\bibliography{references}

\vskip 1 cm
\begin{tabular}{ll}
Chang-Hua JIAO & \\
Department of Mathematics &  \\
Tsinghua University &  \\
Beijing 100084 &  \\
People's Republic of China &  \\
E-mail: jch23@mails.tsinghua.edu.cn &
\end{tabular}

\end{document}